\documentclass[11pt]{article}
\usepackage{amsmath}
\usepackage{amssymb}
\usepackage{amsthm}
\usepackage{amsfonts}
\usepackage{booktabs}

\usepackage{tikz}
\usetikzlibrary{arrows,positioning,matrix, fit, shapes}
\usepackage{tikz-cd}
\usepackage[all]{xy}
\usepackage[margin=2.5cm]{geometry}
\usepackage{enumerate}

\usepackage[colorlinks=true,linkcolor=blue,citecolor=red]{hyperref}
\usepackage{enumitem}
\setenumerate[1]{itemsep=0pt,partopsep=0pt,parsep=\parskip,topsep=3pt}
\setitemize[1]{itemsep=0pt,partopsep=0pt,parsep=\parskip,topsep=3pt}
\setdescription{itemsep=0pt,partopsep=0pt,parsep=\parskip,topsep=3pt}
\setlist[itemize]{leftmargin=35pt}
\setlist[enumerate]{leftmargin=35pt}

\newcommand{\modcat}{\operatorname{mod}}

\newcommand{\Fac}{\operatorname{Fac}}
\newcommand{\Sub}{\operatorname{Sub}}
\newcommand{\add}{\mathsf{add}}
\newcommand{\Hom}{\operatorname{Hom}}
\newcommand{\Ext}{\operatorname{Ext}}

\def\supp{\mathop{\rm Supp}\nolimits}

\newcommand{\PP}{\mathsf{P}}
\newcommand{\II}{\mathsf{I}}

\usepackage{ulem}
\begin{document}
	
	\newcommand{\nc}{\newcommand}
	
	
	
	\newtheorem{theorem}{Theorem}[section]
	\newtheorem{proposition}[theorem]{Proposition}
	\newtheorem{lemma}[theorem]{Lemma}
	\newtheorem{corollary}[theorem]{Corollary}
	\newtheorem{conjecture}[theorem]{Conjecture}
	\newtheorem{question}[theorem]{Question}
	\newtheorem{definition}[theorem]{Definition}
	\newtheorem{example}[theorem]{Example}
	
	\newtheorem{remark}[theorem]{Remark}
	\def\Pf#1{{\noindent\bf Proof}.\setcounter{equation}{0}}
	\def\>#1{{ $\Rightarrow$ }\setcounter{equation}{0}}
	\def\<>#1{{ $\Leftrightarrow$ }\setcounter{equation}{0}}
	\def\bskip#1{{ \vskip 20pt }\setcounter{equation}{0}}
	\def\sskip#1{{ \vskip 5pt }\setcounter{equation}{0}}
	\def\bg#1{\begin{#1}\setcounter{equation}{0}}
		\def\ed#1{\end{#1}\setcounter{equation}{0}}
	\def\KET{T^{^F\bot}\setcounter{equation}{0}}
	\def\KEC{C^{\bot}\setcounter{equation}{0}}
	
	\renewcommand{\thefootnote}{\fnsymbol{footnote}}
	\setcounter{footnote}{0}

	\title{\bf A characterization of IE-closed subcategories   via $\tau$-tilting theory }
	\footnotetext{
		E-mail addresses: hpgao@ahu.edu.cn(Hanpeng Gao);ldjnnu2017004@163.com (Dajun Liu); yzliu3@163.com (Yu-Zhe Liu).}
	\smallskip
	\author{\small  Hanpeng Gao$^a$, Dajun Liu$^{b}\thanks{Corresponding author}$, Yu-Zhe Liu$^c$\\
		{\it \footnotesize $^a$School of Mathematical Sciences, Anhui University, Hefei 230601, China}\\{\it \footnotesize $^b$School of Mathematics and Physics, Anhui Polytechnic University, Wuhu, China }\\{\it \footnotesize $^c$School of Mathematics and Statistics, Guizhou University,  Guizhou,  China}
	}
	\date{}
	
	\maketitle
	\baselineskip 15pt
	%
	%
	\begin{abstract}
Enomoto and Sakai classified functorially finite IE-closed subcategories
over hereditary algebras in terms of twin rigid modules. Their approach
uses the hereditary assumption essentially and therefore does not extend
directly to arbitrary finite-dimensional algebras. In this paper, we
introduce canonical twin support $\tau$-tilting modules and prove that,
for an arbitrary finite-dimensional algebra, they are in bijection with
left-and-right finite IE-closed subcategories, namely those whose generated torsion
and torsion-free classes are both functorially finite. We further give
a characterization of canonicality via the torsion-pair
decompositions associated with $\operatorname{Fac} M$ and
$\operatorname{Sub} N$, which yields a canonicalization procedure whenever
the associated IE-closed subcategory is left-and-right finite. We also introduce
canonical Ext-pairs. If the algebra is hereditary or $\tau$-tilting
finite, then functorially finite IE-closed subcategories are in bijection
with isomorphism classes of canonical Ext-pairs, where the corresponding pair is given by the basic Ext-progenerator and the basic Ext-injective cogenerator.
In the hereditary case, this recovers the twin rigid classification of
Enomoto and Sakai.
		\vspace{10pt}
		
		\noindent {\bf 2020 Mathematics Subject Classification}: 16G10, 16G30.
		
		
		\noindent {\bf Keywords}: IE-closed subcategories, canonical twin support $\tau$-tilting modules,  twin rigid modules, canonical Ext-pairs.

	\end{abstract}
	%
	\vskip 30pt

	\section{Introduction}
	
	The classification of subcategories in module categories is a fundamental problem in the representation theory of Artin algebras. Classical objects of study include torsion classes, which are closed under quotients and extensions, and torsion-free classes, which are closed under submodules and extensions. Wide subcategories, closed under kernels, cokernels, and extensions, also play a central role, deeply connecting to localizations \cite{MS17} and semibricks \cite{Asai20}. Recently, Enomoto and Sakai studied ICE-closed subcategories (closed under images, cokernels, and extensions) \cite{Enomoto22} as a generalization of torsion classes, and then proved that  ICE-closed subcategories are exactly torsion classes over some wide subcategories. By relaxing the cokernel condition, Enomoto and Sakai subsequently defined IE-closed subcategories (closed under images and extensions) \cite{ES23}. It is established that a subcategory is IE-closed if and only if it is the intersection of a torsion class and a torsion-free class.
	
	Parallel to the study of subcategories, Adachi, Iyama, and Reiten introduced support $\tau$-tilting modules \cite{AIR14}. This concept resolves the mutation deficiency of classical tilting modules and provides a bijective correspondence for functorially finite torsion classes.
	
	For IE-closed subcategories, Enomoto and Sakai achieved a classification over hereditary algebras \cite[Theorem 2.14]{ES23} by establishing a bijection between functorially finite IE-closed subcategories and twin rigid modules (pairs of rigid modules $(P, I)$ satisfying specific conditions). However, the twin rigid module classification relies strictly on the hereditary property. Over a hereditary algebra, the vanishing of second extension groups ensures that \(\operatorname{Sub} I\) is extension-closed for a rigid module \(I\). For a general finite-dimensional algebra, non-vanishing second extensions may obstruct this property. Thus, the twin-rigid construction does not extend directly beyond the hereditary setting.
	
	In this paper, we establish a \(\tau\)-tilting-theoretic classification of left-and-right finite IE-closed subcategories over arbitrary finite-dimensional algebras. We replace twin rigid modules with support $\tau$-tilting modules, and support $\tau^-$-tilting modules. We introduce the notion of  canonical twin support $\tau$-tilting modules  and  canonical Ext-pairs (see Definitions \ref{3.2} and \ref{3.9}). Based on this concept, we mainly obtain the following result:

	\vspace{0.2cm}\noindent
	\textbf{Theorem A} (Theorem \ref{3.3} and \ref{3.10}) \textit{Let $\Lambda$ be a finite-dimensional algebra.
		There is a bijection between the set of left-and-right finite
		IE-closed subcategories of $\operatorname{mod}\Lambda$ and the
		set of isomorphism classes of canonical twin support
		$\tau$-tilting modules.
		Moreover, if $\Lambda$ is hereditary or $\tau$-tilting finite,
		then the functorially finite IE-closed subcategories are in
		bijection with the isomorphism classes of canonical Ext-pairs.
}
	\vspace{1mm}
	
	Under the assumptions of the second statement, the canonical Ext-pair corresponding to the left-and-right finite IE-closed subcategory $\mathcal{C}$ is exactly the Ext-progenerator and Ext-injective cogenerator of $\mathcal{C}$. The left-and-right finite IE-closed subcategories must be  functorially finite, moreover, they are  same  for  hereditary algebras.  Hence,
	Ext-pairs and twin rigid modules coincide in the case of hereditary algebras. Theorem A provides the generalization of Enomoto and Sakai's \cite[Theorem 2.14]{ES23}.

	Throughout this paper,  let $k$ be an algebraically closed field and let $\Lambda$ be a finite-dimensional $k$-algebra. 
	All modules considered are finitely generated right $\Lambda$-modules, and $\operatorname{mod}\Lambda$ denotes the category of such modules. All subcategories are assumed to be full and closed under isomorphisms.
	For a module $M\in \operatorname{mod}\Lambda$, we denote by $\mathsf{add} M$ the subcategory consisting of direct summands of finite direct sums of $M$. We denote by $\mathsf{Fac} M$ (resp., $\mathsf{Sub} M$) the subcategory consisting of quotient modules (resp., submodules) of modules in $\mathsf{add} M$. We also  denote by $|M|$ the number of pairwise nonisomorphic indecomposable summands of $M$.

	\section{Preliminaries}

	Let $\mathcal{C}$ be a subcategory of $\operatorname{mod}\Lambda$. A morphism $f \colon M \to C$ with $C \in \mathcal{C}$ is called a \textit{left $\mathcal{C}$-approximation} of $M$ if any morphism from $M$ to an object in $\mathcal{C}$ factors through $f$. The subcategory $\mathcal{C}$ is called \textit{covariantly finite} in $\operatorname{mod}\Lambda$ if every module in $\operatorname{mod}\Lambda$ admits a left $\mathcal{C}$-approximation. Dually, the notions of \textit{right $\mathcal{C}$-approximations} and \textit{contravariantly finite} subcategories are defined. A subcategory is \textit{functorially finite} if it is both covariantly and contravariantly finite \cite{AS81}.

	A subcategory $\mathcal{T}$ of $\operatorname{mod}\Lambda$ is called a \textit{torsion class} if it is closed under quotient modules and extensions. Dually, a subcategory $\mathcal{F}$ is called a \textit{torsion-free class} if it is closed under submodules and extensions.
	
	For a subcategory $\mathcal{X} \subseteq \operatorname{mod}\Lambda$, we define its right and left orthogonal subcategories as follows:
	\begin{align*}
		\mathcal{X}^\perp &= \{ Y \in \operatorname{mod}\Lambda \mid \operatorname{Hom}_\Lambda(X, Y) = 0 \text{ for all } X \in \mathcal{X} \}, \\
		^\perp\mathcal{X} &= \{ Y \in \operatorname{mod}\Lambda \mid \operatorname{Hom}_\Lambda(Y, X) = 0 \text{ for all } X \in \mathcal{X} \}.
	\end{align*}
	
	A pair of subcategories $(\mathcal{T}, \mathcal{F})$ is called a \textit{torsion pair} if $\mathcal{T} = {}^\perp\mathcal{F}$ and $\mathcal{F} = \mathcal{T}^\perp$. In this case, $\mathcal{T}$ is necessarily a torsion class and $\mathcal{F}$ is a torsion-free class. We denote the set of all torsion (resp., torsion-free) classes in $\operatorname{mod}\Lambda$ by $\operatorname{tors}\Lambda$  (resp., $\operatorname{torf}\Lambda$ ).  For a subcategory $\mathcal{X}$, we denote by $\mathcal{T}(\mathcal{X})$ the smallest torsion class containing $\mathcal{X}$, and $\mathcal{F}(\mathcal{X})$  the smallest torsion-free class containing $\mathcal{X}$.

	Let $\tau$ denote the Auslander-Reiten translation in $\operatorname{mod}\Lambda$. Following \cite{AIR14}, a module $M$ is called \textit{$\tau$-tilting} if $\operatorname{Hom}_\Lambda(M, \tau M) = 0$ and $|M|=|\Lambda|$. $M$ is called a \textit{support $\tau$-tilting module} if it is a $\tau$-tilting $\Lambda/\Lambda e\Lambda$-module for some idempotent $e$ of $\Lambda$. Dually, we have the concepts of \textit{$\tau^{-}$-tilting  modules} and \textit{support $\tau^{-}$-tilting modules}.

	The foundation of our algebraic classification relies on the following  bijection.
	\begin{theorem}[\cite{AIR14}]\label{2.1}
		Let $\Lambda$ be an algebra. There is a bijection between the set of isomorphism classes of basic support $\tau$-tilting $\Lambda$-modules and the set of functorially finite torsion classes in $\operatorname{mod}\Lambda$. The bijection is explicitly given by $M \mapsto \mathsf{Fac} M$. Dually, there is a bijection between the set of isomorphism classes of basic support $\tau^{-}$-tilting $\Lambda$-modules and the set of functorially finite torsion-free classes in $\operatorname{mod}\Lambda$  given by $N \mapsto \mathsf{Sub} N$.
	\end{theorem}
	
	An Artin algebra $\Lambda$ is called \textit{$\tau$-tilting finite} if there are only finitely many isomorphism classes of basic support $\tau$-tilting modules. The following result characterizes this finiteness property.
	\begin{theorem}[\cite{DIJ19}]\label{2.2}
		For an Artin algebra $\Lambda$, the following conditions are equivalent:
		\begin{enumerate}[label=\text{\rm(\arabic*)}]
			\item $\Lambda$ is $\tau$-tilting finite.
			\item  $\operatorname{tors}\Lambda$ is a finite set.
			\item Every torsion class in $\operatorname{mod}\Lambda$ is functorially finite.
			\item  $\operatorname{torf}\Lambda$ is a finite set.
			\item Every torsion-free class in $\operatorname{mod}\Lambda$ is functorially finite.
		\end{enumerate}
	\end{theorem}

	\begin{definition}[\cite{ES23}]\label{2.3}
		A subcategory $\mathcal{C}$ of $\operatorname{mod}\Lambda$ is called IE-closed if it is closed under images and extensions.
	\end{definition}
	
	The following basic property characterizes IE-closed subcategories via torsion theory.
	\begin{proposition}[\cite{ES23}]\label{2.4}
		A subcategory $\mathcal{C}$ is IE-closed if and only if there exist a torsion class $\mathcal{T}$ and a torsion-free class $\mathcal{F}$ such that $\mathcal{C} = \mathcal{T} \cap \mathcal{F}$. In particular, for any IE-closed subcategory $\mathcal{C}$, it intrinsically satisfies $\mathcal{C} = \mathcal{T}(\mathcal{C}) \cap \mathcal{F}(\mathcal{C})$.
	\end{proposition}
	
	\section{Main  Results}
	
We call an IE-closed subcategory $ \mathcal{C}$  left-and-right finite if both $ \mathcal{T}(\mathcal{C})$ and  $ \mathcal{F}(\mathcal{C})$ are functorially finite. In this section, we establish a characterization of left-and-right finite IE-closed subcategories for algebras.

	\begin{proposition}\label{3.1}
		Let $\mathcal{C} = \mathcal{T} \cap \mathcal{F}$ be an IE-closed subcategory in $\operatorname{mod}\Lambda$, where   $\mathcal{T}$ is a torsion class and $\mathcal{F}$ is a torsion-free class. If both $\mathcal{T}$ and $\mathcal{F}$ are functorially finite, then $\mathcal{C}$ is a functorially finite subcategory.
	\end{proposition}
	
	\begin{proof}
		First, let $M \in \operatorname{mod}\Lambda$. Since $\mathcal{F}$ is contravariantly finite, we take a right $\mathcal{F}$-approximation $g \colon F \to M$. Next, since $\mathcal{T}$ is a torsion class, the   right $\mathcal{T}$-approximation of $F$ is precisely the inclusion of its torsion radical $h \colon t_{\mathcal{T}}(F) \hookrightarrow F$.
		We claim that the composition $g \circ h \colon t_{\mathcal{T}}(F) \to M$ is a right $\mathcal{C}$-approximation of $M$. Indeed, $h$ is a monomorphism meaning $t_{\mathcal{T}}(F)$ is a submodule of $F$. Since $\mathcal{F}$ is closed under submodules, we have  $t_{\mathcal{T}}(F) \in \mathcal{F}$. Hence, $t_{\mathcal{T}}(F) \in \mathcal{T} \cap \mathcal{F} = \mathcal{C}$.
		Let $C \in \mathcal{C}$ and $\alpha \colon C \to M$ be an arbitrary morphism. Since $C \in \mathcal{F}$ and $g$ is a right $\mathcal{F}$-approximation, $\alpha$ factors through $g$ via some $\alpha' \colon C \to F$. Furthermore, since $C \in \mathcal{T}$ and $h$ is a right $\mathcal{T}$-approximation, $\alpha'$ factors through $h$ via some $\beta \colon C \to t_{\mathcal{T}}(F)$. This yields the following commutative diagram:
		\[
		\xymatrix{
			C \ar[rrd]^{\alpha} \ar[rd]|{\alpha'} \ar@{-->}[d]_{\beta} & & \\
			t_{\mathcal{T}}(F) \ar@{^{(}->}[r]_{h} & F \ar[r]_{g} & M
		}
		\]
		Thus, $\alpha = (g \circ h) \circ \beta$.	Therefore, $g \circ h$ serves as a right $\mathcal{C}$-approximation, and $\mathcal{C}$ is contravariantly finite.
		
		Dually, we can prove that $\mathcal{C}$ is covariantly finite.
		So it is functorially finite.

	\end{proof}

Hence, 	every left-and-right finite IE-closed subcategory is always functorially finite. Moreover, they are  the same  for  hereditary algebras by \cite[Lemma 2.9]{ES23}.  If $\Lambda$ is $\tau$-tilting finite,  there are only finitely many torsion classes in $\operatorname{mod}\Lambda$ by Theorem \ref{2.2}, all of which are functorially finite. Thus, every IE-closed subcategory is left-and-right finite and hence  functorially finite.

	\begin{definition}\label{3.2}
		Let $\Lambda$ be an algebra. A pair of modules $(M, N)$ is called a \textbf{twin support $\tau$-tilting module} if  $M$ is a basic support $\tau$-tilting module and $N$ is a basic support $\tau^{-}$-tilting module. A	twin support $\tau$-tilting module $(M, N)$ is called  \textbf{canonical} if  it satisfies 	$	\mathrm{Fac}(M) = \mathcal{T}(\mathrm{Fac}(M) \cap \mathrm{Sub}(N))$ and $
		\mathrm{Sub}(N) = \mathcal{F}(\mathrm{Fac}(M) \cap \mathrm{Sub}(N)).$
		
	\end{definition}

	\begin{theorem} \label{3.3}
		Let $\Lambda$ be an   algebra. There is a bijection between the set of left-and-right finite IE-closed subcategories of $\mathrm{mod}\, \Lambda$ and  the set of isomorphism classes of canonical twin support $\tau$-tilting modules $(M, N)$ in $\mathrm{mod}\, \Lambda$. The bijection is explicitly given by the mappings:
		\begin{align*}
			\Phi: (M, N) &\mapsto \mathrm{Fac}(M) \cap \mathrm{Sub}(N), \\
			\Psi: \mathcal{C} &\mapsto (M_{\mathcal{C}}, N_{\mathcal{C}}),
		\end{align*}
		where $M_{\mathcal{C}}$ is the basic support $\tau$-tilting module generating $\mathcal{T}(\mathcal{C})$, and $N_{\mathcal{C}}$ is the basic support $\tau^{-}$-tilting module cogenerating $\mathcal{F}(\mathcal{C})$.
	\end{theorem}
	
	\begin{proof}
		Firstly, let $(M, N)$ be a canonical twin support $\tau$-tilting module.
		Theorem \ref{2.1}  implies $ \mathrm{Fac}(M)$ is a functorially finite torsion class and $\mathrm{Sub}(N)$ is a functorially finite torsion-free class.  By the  definition of canonical twin support $\tau$-tilting modules $(M, N)$, we have:
		\[
		\mathcal{T}(\mathcal{C}) = \mathcal{T}(\mathrm{Fac}M \cap \mathrm{Sub}N) = \mathrm{Fac}M,~~\text{and}~~~
		\mathcal{F}(\mathcal{C}) = \mathcal{F}(\mathrm{Fac}M \cap \mathrm{Sub}N) = \mathrm{Sub}N.
		\]
	 Proposition \ref{2.4} implies $\mathcal{C} = \Phi(M, N) = \mathrm{Fac}M \cap \mathrm{Sub}N$ is a left-and-right finite IE-closed subcategory.
	 Since $M$ is a basic support $\tau$-tilting module generating the same torsion class as $M_{\mathcal{C}}$, the uniqueness of support $\tau$-tilting modules ensures that $M_{\mathcal{C}} \cong M$. Dually, we have $N_{\mathcal{C}} \cong N$. Thus, $\Psi(\Phi(M, N)) = (M, N)$.
		
		Secondly, let $\mathcal{C}$ be a left-and-right finite IE-closed subcategory. Then both the torsion class $\mathcal{T}(\mathcal{C})$ and torsion-free class $\mathcal{F}(\mathcal{C})$ are functorially finite. Thus, there exists a unique basic support $\tau$-tilting module $M_{\mathcal{C}}$ such that $\mathrm{Fac}M_{\mathcal{C}} = \mathcal{T}(\mathcal{C})$, and a unique basic support $\tau^{-}$-tilting module $N_{\mathcal{C}}$ such that $\mathrm{Sub}N_{\mathcal{C}} = \mathcal{F}(\mathcal{C})$.
		By definition, $\Phi(M_{\mathcal{C}}, N_{\mathcal{C}}) = \mathrm{Fac}M_{\mathcal{C}} \cap \mathrm{Sub}N_{\mathcal{C}} = \mathcal{T}(\mathcal{C}) \cap \mathcal{F}(\mathcal{C})$. Proposition \ref{2.4} implies $\mathcal{C} = \mathcal{T}(\mathcal{C}) \cap \mathcal{F}(\mathcal{C})$. Hence, $\Phi(\Psi(\mathcal{C})) = \mathcal{C}$. Moreover,  $\mathcal{C} = \mathrm{Fac}M_{\mathcal{C}} \cap \mathrm{Sub}N_{\mathcal{C}}$  yields:
		\begin{align*}
			\mathcal{T}(\mathrm{Fac}M_{\mathcal{C}} \cap \mathrm{Sub}N_{\mathcal{C}}) &= \mathcal{T}(\mathcal{C}) = \mathrm{Fac}M_{\mathcal{C}}, \\
			\mathcal{F}(\mathrm{Fac}M_{\mathcal{C}} \cap \mathrm{Sub}N_{\mathcal{C}}) &= \mathcal{F}(\mathcal{C}) = \mathrm{Sub}N_{\mathcal{C}}.
		\end{align*}
		This confirms that $(M_{\mathcal{C}}, N_{\mathcal{C}})$ is indeed a canonical twin support $\tau$-tilting module.
		
		Finally, the bijection is established.
	\end{proof}

	\begin{remark}\label{rmk:canonical_necessity}
		Without the ``canonical'' conditions, the map 	$\Phi(M, N)= \mathrm{Fac}(M) \cap \mathrm{Sub}(N)$ is  not injective. Let $\mathcal{C} = \{0\}$ be the trivial IE-closed subcategory  in $\modcat\Lambda$. For any support $\tau$-tilting module $M$, the image of the twin support $\tau$-tilting module $(M,0)$ under $\Phi$ has always been $\mathcal{C}$.
		However,  $(M,0)$ fails to be canonical  unless $M$ equals 0.
		
	\end{remark}

	For $M\in\modcat \Lambda$, let $\supp M:=\{i|\Hom_{\Lambda}(P_i, M)\neq 0\}$ which is called the support of $M$. In addition, the  support of
	a subcategory 
	$\mathcal{C}$  is  $\supp \mathcal{C}:=\bigcup\limits_{M\in\mathcal{C}}\supp M$. It is clear $\supp \mathcal{C}=\emptyset$ if and only if $\mathcal{C}=\{0\}$.
	The following results are  obvious:
	$$\supp \Fac M=\supp M\quad\text{and}\quad\supp \Sub N=\supp N,$$
	$$\supp \mathcal{C}=\supp \mathcal{T}(\mathcal{C})=\supp \mathcal{F}(\mathcal{C}).$$
	Thus,   a canonical twin support $\tau$-tilting module $(M, N)$ is necessary to satisfy the condition $\supp M=\supp N$.

	\begin{remark} This mapping $\Psi$ is different from the one given by Enomoto and Sakai in \cite{ES23}, even in terms of  hereditary algebras. They obtained the Ext-projective object and the Ext-injective object within the IE-closed subcategory $\mathcal{C} $, while we obtain  the Ext-projective  (resp., Ext-injective) object within the torsion class $ \mathcal{T}(\mathcal{C})$ (resp., torsion-free class $ \mathcal{F}(\mathcal{C}) $). For example, let $\Lambda=k(2\to 1)$ and $\mathcal{C} = \mathsf{add}\{P_2\}$ be an IE-closed subcategory. Then it corresponds to twin rigid module $(P_2,P_2)$ and canonical twin support $\tau$-tilting module $(P_2\oplus S_2, P_2\oplus S_1)$. 
	\end{remark}


	An IE-closed subcategory of  $\operatorname{mod}\Lambda$  is called ICE-closed if $\mathsf{Cok}(\mathcal{C})\subseteq\mathcal{C} $ where  $\mathsf{Cok}(\mathcal{C})=\{Z|X\to Y\to Z\to 0,X,Y\in \mathcal{C}\} $. Dually, the notions for IKE-closed subcategories and $\mathsf{Ker}(\mathcal{C})$ are defined.
	By imposing specific conditions on the canonical twin support $\tau$-tilting module  $(M, N)$, we can describe some specific IE-closed subcategories.

	\begin{proposition} \label{cor:unification}
		Let $\Lambda$ be an algebra. Assume that  $\mathcal{C}$ is a left-and-right finite IE-closed subcategory which corresponds to the canonical twin support $\tau$-tilting module  $(M, N)$, then
		
		\begin{enumerate}[label=\text{\rm(\arabic*)}]
			\item
			$\mathcal{C}$ is a torsion class if and only if $  \mathsf{Fac} M\subseteq \mathsf{Sub}N$.
			
			\item
			$\mathcal{C}$ is a torsion-free class if and only if  $\mathsf{Sub}{N}\subseteq \mathsf{Fac} M$.
			
			\item $\mathcal{C}$ is  ICE-closed  if and only if $\mathsf{Cok}(\mathcal{C}) \subseteq \mathsf{Sub}N$.

			\item  $\mathcal{C}$ is  IKE-closed if and only if $\mathsf{Ker}(\mathcal{C}) \subseteq \mathsf{Fac} M$.
			
		\end{enumerate}
	\end{proposition}
	
	\begin{proof}
		Based on the assumption, we have 
		$$\mathcal{C}=\mathrm{Fac}M \cap \mathrm{Sub}N\quad \text{and}\quad\mathcal{T}(\mathcal{C}) = \mathcal{T}(\mathrm{Fac}M \cap \mathrm{Sub}N) = \mathrm{Fac}M.$$
		Hence, $\mathcal{C}$ is a torsion class if and only if $\mathcal{C}=\mathcal{T}(\mathcal{C}) $. This is equivalent to  $  \mathsf{Fac} M\cap \mathrm{Sub}N= \mathsf{Fac}M$. Thus, (1) has been obtained. 
		
		Since $\mathsf{Fac} M$ is a torsion class,   $\mathsf{Cok}(\mathcal{C})$ necessarily lies in $\mathsf{Fac} M$. Therefore, the condition that $\mathcal{C}$ is ICE-closed is equivalent to  $\mathsf{Cok}(\mathcal{C}) \subseteq\mathrm{Sub}N$ which implies (3).
		
		Similarly, we  can prove (2)  and (4).\end{proof}

	While the definition of a canonical twin support $\tau$-tilting module is conceptually symmetric and globally determined by the intersection $\mathcal{C}$, it is not as intuitive as the definition of the twin rigid module. Next, 
	we provide this correct characterization of  canonical twin support $\tau$-tilting modules via  torsion theory.
	
	Let $M$ be a support $\tau$-tilting module and $N$ a support $\tau^-$-tilting module in  $\mathrm{mod}\, \Lambda$.
	The torsion-free class $\mathrm{Sub}N$ induces a torsion pair $({}^{\perp}N, \mathrm{Sub}N)$ in $\mathrm{mod}\, \Lambda$. Thus, there exists a canonical short exact sequence for $M$:
	\begin{equation} \label{eq:seq_M}
		0 \longrightarrow U_M \longrightarrow M \xrightarrow{\pi} P_M \longrightarrow 0,
	\end{equation}
	where $P_M \in \mathrm{Sub}N$ is the maximal torsion-free quotient of $M$, and $U_M \in {}^{\perp}N$.
	Dually, the torsion class $\mathrm{Fac}M$ induces a torsion pair $(\mathrm{Fac}M, M^{\perp})$. The canonical short exact sequence for $N$ is:
	\begin{equation} \label{eq:seq_N}
		0 \longrightarrow I^N \xrightarrow{\iota} N \longrightarrow V^N \longrightarrow 0,
	\end{equation}
	where $I^N \in \mathrm{Fac}M$ is the maximal torsion submodule of $N$, and $V^N \in M^{\perp}$.

	\begin{theorem} \label{3.7}
		Let $M$ be a basic support $\tau$-tilting module and $N$ be a basic support $\tau^{-}$-tilting module. Let $\mathcal{C} = \mathrm{Fac}M \cap \mathrm{Sub}N$. Then the following equalities hold:
		
		\begin{enumerate}[label=\text{\rm(\arabic*)}]
			
			\item 
			$\mathcal{T}(\mathcal{C}) = \mathcal{T}(P_M)$.
			\item
			$\mathcal{F}(\mathcal{C}) = \mathcal{F}(I^N)$.
			
		\end{enumerate}
		
		Consequently, the pair $(M, N)$ is  canonical  if and only if $\mathrm{Fac}M = \mathcal{T}(P_M)$ and $\mathrm{Sub}N = \mathcal{F}(I^N)$.
	\end{theorem}
	
	\begin{proof}
		We first prove $\mathcal{T}(\mathcal{C}) = \mathcal{T}(P_M)$.
		
		Note that $P_M$ is a quotient of $M$ by the sequence (\ref{eq:seq_M}), we have  $P_M \in \mathrm{Fac}M$. By construction, $P_M \in \mathrm{Sub}N$. Therefore, $P_M \in \mathrm{Fac}M \cap \mathrm{Sub}N = \mathcal{C}$ yields $\mathcal{T}(P_M) \subseteq \mathcal{T}(\mathcal{C})$.
		
		To show the reverse inclusion $\mathcal{T}(\mathcal{C}) \subseteq \mathcal{T}(P_M)$, let $X \in \mathcal{C}$. Since $X \in \mathrm{Fac}M$, there exists an integer $d \ge 1$ and a surjective morphism $f: M^d \twoheadrightarrow X$. Consider  the following diagram:
		
		\begin{center}	
			\begin{tikzcd}
				0 \arrow[r]  & U^d_M \arrow[r] \arrow[rd, "0"] & M^d \arrow[r]\arrow[d, two heads, "f"]  & P^d_M \arrow[r] \arrow[ld, dashed, "k"] & 0  \\
				&  &X  &   & 
			\end{tikzcd}
			
		\end{center}		
		Since $X \in \mathcal{C} \subseteq \mathrm{Sub}N$, 
		we know every morphism from $U^d_M$ to $X$ is zero.
		Therefore, any morphism from $M^d$ to $X$ must factor  through the maximal torsion-free quotient of $M^d$, which is precisely $P_M^d$. This factorization yields a surjective morphism $k: P_M^d \twoheadrightarrow X$. Consequently, $X \in \mathrm{Fac}P_M \subseteq \mathcal{T}(P_M)$. Since this holds for all $X \in \mathcal{C}$, we obtain $\mathcal{C} \subseteq \mathcal{T}(P_M)$. Taking the torsion closure on both sides gives $\mathcal{T}(\mathcal{C}) \subseteq \mathcal{T}(P_M)$. Thus, $\mathcal{T}(\mathcal{C}) = \mathcal{T}(P_M)$.
		
		By a completely dual argument, we can prove $\mathcal{F}(\mathcal{C}) = \mathcal{F}(I^N)$.
		
		Finally, by definition of the canonical support $\tau$-tilting module, $(M, N)$ is canonical if and only if $\mathrm{Fac}M = \mathcal{T}(P_M)$ and $\mathrm{Sub}N = \mathcal{F}(I^N)$.
	\end{proof}
	
	\vspace{0.2cm}\noindent
	{\bf A canonicalization procedure for twin support $\tau$-tilting modules}:
	
	The   characterization for canonical  twin support $\tau$-tilting module in Theorem \ref{3.7} provides a constructive algorithm to canonicalize a twin support $\tau$-tilting module  $(M, N)$ such that $\mathcal C_{M,N}:=\operatorname{Fac}M\cap\operatorname{Sub}N$
	is left-and-right finite. Given $(M, N)$, we compute $P_M$ and $ I^N$ via the sequence (\ref{eq:seq_M})  and (\ref{eq:seq_N}). Let $M^*$ be the basic support $\tau$-tilting module  and $N^*$  the basic support $\tau^-$-tilting module such that $\mathsf{Fac} M^* = \mathcal{T}(P_M)$ and $\mathsf{Sub} N^* = \mathcal{F}(I^N)$. Then $(M^*, N^*)$ is the unique canonical twin support $\tau$-tilting module such that $\Phi(M^*, N^*) = \mathcal{C}_{M,N}$.

	\vspace{0.2cm}

	Let $\mathcal{C}$ be a subcategory of $\mathrm{mod}\, \Lambda$. $P\in\mathcal{C}$ is Ext-projective if it satisfies $\Ext^1
	(P, C) = 0$ for any $C\in\mathcal{C}$.
	$P\in\mathcal{C}$  is an Ext-progenerator of $\mathcal{C}$ if $P$ is Ext-projective  and for any $W\in\mathcal{C}$, there exists a short exact sequence
			$0\to L\to P'\to W\to 0$
	such that $P'\in \mathsf{add} P$  and $L\in\mathcal{C}$.
	If $\mathcal{C}$ has an Ext-progenerator, then $\PP( \mathcal{C})$ denotes a unique basic Ext-progenerator of  $\mathcal{C}$,
	that is, a direct sum of all indecomposable Ext-projective objects in  $\mathcal{C}$ up to isomorphism.
	Dually, the notions for Ext-injectives are defined, and $\II( \mathcal{C})$ denotes a unique basic Ext-injective
	cogenerator of  $\mathcal{C}$ (if exists).
	
	Now, let $X \in\mathcal{C} = \mathrm{Fac}M \cap \mathrm{Sub}N$. Applying $\Hom_{\Lambda}(-,X)$ to the exact sequence (\ref{eq:seq_M}), we obtain
	$$\Hom_{\Lambda}(U_M,X)\to\Ext^1(P_M,X)\to \Ext^1(M,X).$$
	Note that  $U_M \in {}^{\perp}N={}^{\perp}\Sub N$, we have $\Hom_{\Lambda}(U_M,X)=0$. Since $M$ is a support $\tau$-tilting module, we have $\Ext^1(M,\Fac M)=0$ by  \cite[Proposition~1.2(a)]{AIR14} which implies $\Ext^1(M,X)=0$. Thus, $\Ext^1(P_M,X)=0$ and $P_M$ is Ext-projective in $\mathcal{C} = \mathrm{Fac}M \cap \mathrm{Sub}N$. We claim  that $P_M$ is the  Ext-progenerator of  $\mathcal{C}$. 
	
	Let $M$ be the  $\tau$-tilting $\Lambda/\Lambda e\Lambda$-module for some idempotent $e$ and $W\in\mathcal{C}$. Since $W\in\Fac M$, there is a surjective right $\add M$-approximation  $g: M' \twoheadrightarrow W$ with $M' \in \add M$, yielding a short exact sequence in $\modcat\,\Lambda$:
	\begin{equation}
		0 \to K \to M' \xrightarrow{g} W \to 0 \label{eq:cover}
	\end{equation}
	\cite[Lemma 2.6]{AIR14} shows $K\in {}^{\perp}(\tau M)$. Note that ${}^{\perp}(\tau M)\cap (e\Lambda)^{\perp}=\Fac M$\cite[Corollary 2.13]{AIR14} and $K\in  (e\Lambda)^{\perp}$, we have $K \in \Fac M$.
	On the other hand, we have the canonical decomposition sequence for $M' \in \add M$:
	\begin{equation}
		0 \to U_{M'} \to M' \xrightarrow{\pi} P_{M'} \to 0 \label{eq:canonical2}
	\end{equation}
	where $U_{M'} \in {}^\perp N$ and $P_{M'} \in \Sub N$. Note that $P_{M'} \in \Fac M$ since it is a quotient of $M' \in \add M$. Thus $P_{M'} \in \Fac M \cap \Sub N = \mathcal{C}$. Since $W \in \mathcal{C}\subseteq \Sub N$ and $U_{M'} \in {}^\perp N$, we have $\Hom_\Lambda(U_{M'}, W) = 0$. 
	By the universal property of cokernels, $g$ factors through $\pi$, inducing a surjective morphism $h: P_{M'} \twoheadrightarrow W$ such that $h \circ \pi = g$. 
	Let $L = \ker h$.
	\begin{center}	
		\begin{tikzcd}
			& 0 \arrow[d] &0\arrow[d]  &0\arrow[d]  &   \\
			0 \arrow[r,dashed]  & U_{M'} \arrow[r,dashed] \arrow[d,equal] &K \arrow[r,dashed]\arrow[d]  &L \arrow[d] &  \\
			0 \arrow[r]  & U_{M'} \arrow[r]\arrow[d]  \arrow[rd, "0"] & M' \arrow[r,"\pi"]\arrow[d, "g"]  & P_{M'} \arrow[r] \arrow[d, dashed, "h"] & 0  \\
			0\arrow[r]	&0 \arrow[r]\arrow[d]  &W \arrow[r,equal]  \arrow[d] &  W \arrow[r]\arrow[d]&0 \\
			& 0 &0  &  0 &
			\arrow[from=2-4, to=5-2, dashed, bend left=30]
		\end{tikzcd}
		
	\end{center}
	By  the Snake Lemma, we have a short exact sequence:
	$$	0 \to U_{M'} \to K \to L \to 0.$$
	Since we established earlier that $K \in \Fac M$, it immediately guarantees that $L \in \Fac M$. By definition of  $L = \ker h$,  $L$ is a submodule of $P_{M'} \in \Sub N$. It inevitably follows that $L \in \Sub N$.
	Therefore, $L \in \Fac M \cap \Sub N = \mathcal{C}$. 
	This implies that the short exact sequence  within $\mathcal{C}$:
	\[ 0 \to L \to P_{M'} \xrightarrow{h} W \to 0. \]
Since the torsion-free quotient functor associated with the torsion pair \(({}^\perp N,\operatorname{Sub}N)\) is additive, $M'\in \add M$ implies $P_{M'}\in \add P_M$. Hence the above exact sequence shows that \(P_M\) is an Ext-progenerator of $\mathcal{C}$. Dually, we can obtain a similar statement for  $I^N$. Finally, we have the following result:
	
	\begin{proposition}\label{3.8} Let $(M,N)$ be a twin support $\tau$-tilting module and $\mathcal{C} = \mathrm{Fac}M \cap \mathrm{Sub}N$. Assume that $P_M$, $I_N$ are defined as sequences \eqref{eq:seq_M} and \eqref{eq:seq_N}. If $P$ is the basic module such that $\add P=\add P_M$ and $I$ is the basic module such that $\add I=\add I^N$, then
		we have	  $\PP( \mathcal{C})=P$  and $\II( \mathcal{C})=I$.
	\end{proposition}
	
	\begin{definition}\label{3.9}
		Let $\Lambda$ be an algebra. A pair of basic modules $(P, I)$ is called a  \textbf{canonical  Ext-pair} if  there exists
		a pair $(M, N)$ satisfying
		\begin{enumerate}[label=\text{\rm(\arabic*)}]
			\item $M$ is a support $\tau$-tilting module and $N$ is a support $\tau^-$-tilting module.

			\item There are two short exact sequences:
			
			$$0 \longrightarrow U_M \longrightarrow M \xrightarrow{} P' \longrightarrow 0,$$
			$$0 \longrightarrow I' \xrightarrow{} N \longrightarrow V^N \longrightarrow 0,$$
			with $P' \in \mathrm{Sub}N$, $U_M \in {}^{\perp}N$ and
			$I' \in \mathrm{Fac}M$, $V^N \in M^{\perp}$. 
			
			\item $\add P =\add P'$  and  $\add I=\add I'$.
		\end{enumerate}

	\end{definition}

	\begin{theorem} \label{3.10}
		Let $\Lambda$ be a   hereditary or $\tau$-tilting finite algebra. There is a bijection between the set of functorially finite IE-closed subcategories of $\mathrm{mod}\, \Lambda$ and  the set of isomorphism classes of canonical Ext-pairs $(P, I)$ in $\mathrm{mod}\, \Lambda$. The bijection is explicitly given by the mappings:
		\begin{align*}
			\Phi': (P, I) &\mapsto  \mathcal{T}(P) \cap  \mathcal{F}(I), \\
			\Psi': \mathcal{C} &\mapsto (\PP({\mathcal{C}}), \II( \mathcal{C})).
		\end{align*}
	\end{theorem}
	
	\begin{proof}
		Let  $M$ be a support $\tau$-tilting module and $N$  a support $\tau^-$-tilting module such that $(P, I)$ is a canonical Ext-pair. Then $\mathcal{C} = \mathrm{Fac}(M) \cap \mathrm{Sub}(N)$ which is a left-and-right finite IE-closed subcategory when  $\Lambda$ is  hereditary or $\tau$-tilting finite. 
		Theorem \ref{3.7}  implies $\mathcal{T}(\mathcal{C})=	\mathcal{T}(P)$ and 
		$\mathcal{F}(\mathcal{C})=	\mathcal{F}(I)$. Hence, $$\Phi' (P, I)=\mathcal{T}(P) \cap  \mathcal{F}(I)=\mathcal{T}(\mathcal{C})\cap\mathcal{F}(\mathcal{C})=\mathcal{C} $$
		is a  functorially finite IE-closed subcategory. Moreover, by Proposition \ref{3.8}
		$$\Psi'\Phi' (P, I)=\Psi'(\mathcal{C})= (\PP({\mathcal{C}}), \II( \mathcal{C}))=(P, I).$$

		Now, let  $\mathcal{C}$ be a functorially finite  IE-closed subcategory. Then  $\mathcal{C}$ is a left-and-right finite  IE-closed subcategory which corresponds to the canonical twin support $\tau$-tilting module  $(M, N)$. Then
		$\mathcal{C} = \mathrm{Fac}(M) \cap \mathrm{Sub}(N)$.  There are two canonical short exact sequences via torsion pair $({}^{\perp}N, \mathrm{Sub}N)$ and $(\mathrm{Fac}M, M^{\perp})$:
		$$0 \longrightarrow U_M \longrightarrow M \xrightarrow{} P' \longrightarrow 0,$$
		$$0 \longrightarrow I' \xrightarrow{} N \longrightarrow V^N \longrightarrow 0,$$
		with $P' \in \mathrm{Sub}N$, $U_M \in {}^{\perp}N$ and 
		$I' \in \mathrm{Fac}M$, $V^N \in M^{\perp}$. We also have $\mathcal{T}(\mathcal{C})=	\mathcal{T}(P')$ and 
		$\mathcal{F}(\mathcal{C})=	\mathcal{F}(I')$  by Theorem \ref{3.7} again.
		Let $P$ be the basic module such that $\add P=\add P'$ and $I$  the basic module such that $\add I=\add I'$.
		Proposition \ref{3.8} shows that   $\PP( \mathcal{C})=P$  and $\II( \mathcal{C})=I$ and, hence
		$$\Psi'(\mathcal{C}) =(\PP({\mathcal{C}}), \II( \mathcal{C}))=(P,I)$$
		is  a canonical Ext-pair.  Moreover,
		$$\Phi'\Psi'(\mathcal{C})=\Phi' (P, I)=\mathcal{T}(P) \cap  \mathcal{F}(I)=\mathcal{T}(P') \cap  \mathcal{F}(I')=\mathcal{T}(\mathcal{C})\cap\mathcal{F}(\mathcal{C})=\mathcal{C}.$$
		
	\end{proof}

In the hereditary case, Theorem \ref{3.10} associates to \(\mathcal C\) the pair \((P(\mathcal C),I(\mathcal C))\). By \cite[Theorem 2.14]{ES23}, the same pair is precisely the twin rigid module corresponding to \(\mathcal C\). Hence canonical Ext-pairs and twin rigid modules coincide.

	\section{Example}
	
	In this section, we illustrate our main results using a concrete example of a non-hereditary algebra.
	\begin{example}\label{ex:nakayama_full_computation}
		Let $\Lambda$ be the Nakayama algebra given by the quiver $1 \xrightarrow{\alpha} 2 \xrightarrow{\beta} 3$ with the relation $\alpha\beta = 0$.
		Note that $\Lambda$ is $\tau$-tilting finite but not hereditary ($\operatorname{gl.dim}\Lambda = 2$).
		There are exactly 12 basic support $\tau$-tilting modules and  12 basic support $\tau^-$-tilting modules. Note that 
		a canonical twin support $\tau$-tilting module $(M, N)$ necessarily satisfies the condition $\supp M=\supp N$, we list them according to their support sets, here we omit the symbols "$\oplus$" (see Table \ref{tab1}).
		\begin{table}[htbp]
			\centering
			\renewcommand{\arraystretch}{1.3} 
			\caption{ Support  $\tau$-(resp.,$\tau^-$-)tilting modules via support sets}
			\label{tab1}
			\begin{tabular}{c|c|c}
				\hline
				\textbf{Support Sets } & \textbf{Support $\tau$-tilting Modules}& \textbf{Support $\tau^-$-tilting Modules }\\
				\hline
				$\emptyset$  & $0$                        & $0$ \\
				\hline
				$\{1\} $  & $S_1$            & $S_1$ \\
				\hline
				$\{2\} $  & $S_2$            & $S_2$ \\
				\hline
				$	\{3\} $ & $S_3$            & $S_3$ \\
				\hline
				$	\{1,3\}$  & $S_1S_3$       & $S_1S_3$  \\
				\hline
				$	\{1,2\}$  &$S_1P_1$,\quad$S_2P_1$  &$S_1P_1$,\quad$S_2  P_1$  \\
				\hline
				$	\{2,3\}$  &$S_3P_2$,\quad$S_2P_2$  & $S_3P_2$,\quad$S_2P_2$  \\
				\hline
				$\{1,2,3\}$&$S_1S_3P_1$,\quad$S_2P_1P_2$,\quad$\Lambda$&$S_1S_3P_2$,\quad$S_2P_1P_2$,\quad$D\Lambda$ \\

				\hline
			\end{tabular}
		\end{table}

		Hence, we have $1^2+1^2+1^2+1^2+1^2+2^2+2^2+3^2=22$   twin support $\tau$-tilting modules with equal support. Since equality of
		supports is necessary for canonicality, all canonical pairs occur
		among these $22$ pairs. We will use Theorem \ref{3.7} to verify
		$(M, N) = (S_1S_3P_1, S_1S_3P_2)$ is a non-canonical twin support $\tau$-tilting module, and then construct the unique canonical twin support $\tau$-tilting module  $(M^*, N^*)$. Clearly,  $ \mathsf{Fac}M=\mathsf{Fac}(S_1S_3P_1) = \mathsf{add}\{S_1,S_3, P_1\}$ and  $\mathsf{Sub}N=\mathsf{Sub}(S_1S_3P_2) = \mathsf{add}\{S_1,S_3,P_2\}$. The decomposition of $M$ with respect to the  torsion pair $({}^{\perp}N, \mathrm{Sub}N)=(\mathsf{add}\{S_2\}, \mathsf{add}\{S_1,S_3,P_2\})$ yields the short exact sequence:
		$$ 0 \longrightarrow S_2 \longrightarrow S_1S_3P_1 \longrightarrow S_1S_1S_3 \longrightarrow 0. $$
		Similarly, $(\mathrm{Fac}M, M^{\perp})=(\mathsf{add}\{S_1,S_3,P_1\}, \mathsf{add}\{S_2\})$ and 	we have the canonical short exact sequence for $N$:	
		$$ 0 \longrightarrow S_1S_3S_3 \longrightarrow S_1S_3P_2 \longrightarrow S_2 \longrightarrow 0. $$	
		Thus,  $P_{M} = S_1S_1S_3$ and  $I_{N} = S_1S_3S_3$.
		Note that $\mathcal{T}(P_{M}) = \mathcal{T}(S_1S_3) = \mathsf{add}\{S_1,S_3\}\neq\mathsf{Fac} M$. Thus, $(M,N)$ is non-canonical.  Now, take $M^* = S_1S_3$ and $N^* = S_1S_3$. It is  easy to verify $\mathsf{Fac} M^* = \mathcal{T}(P_{M})$ and $\mathsf{Sub} N^*=\mathcal{F}(I^N)$.
		Thus, the unique canonical twin support $\tau$-tilting module corresponding to $\mathcal{C} =\mathsf{Fac} M\cap \mathsf{Sub} N= \add\{S_1,S_3\}$ is $(M^*,N^*)=(S_1S_3,S_1S_3)$.
		
		Finally, Finally, the remaining pairs are all  canonical twin support $\tau$-tilting modules.  The following Table \ref{tab2} lists the one-to-one correspondence among the 21 IE-closed subcategories, the 21 canonical twin support $\tau$-tilting modules, and  canonical Ext-pairs.

		\begin{table}[htpb]
			\centering
			\caption{IE-closed subcategories,canonical twin support $\tau$-tilting modules, canonical Ext-pairs}
			\label{tab2}
			\renewcommand{\arraystretch}{1.8}
			\begin{tabular}{c|c|c}
				\hline
				\textbf{IE-closed Subcategory} & \textbf{Canonical Twin Support}  & \textbf{Canonical Ext-pairs} \\
				\textbf{ $\mathcal{C}$}&  \textbf{  $\tau$-tilting Modules $(M, N)$ }& \textbf{(P,I)} \\
				\hline
				$0$ & $(0,0)$ & $(0, 0)$ \\
				$\mathsf{add}\{S_1\}$ & $(S_1,S_1)$ & $(S_1, S_1)$ \\
				$\mathsf{add}\{S_2\}$ & $(S_2,S_2)$& $(S_2,S_2)$ \\
				$\mathsf{add}\{S_3\}$ & $(S_3,S_3)$ & $(S_3,S_3)$ \\
				$\mathsf{add}\{S_1, S_3\}$ & $(S_1S_3,S_1S_3)$ &$(S_1S_3,S_1S_3)$ \\
				$\mathsf{add}\{S_1,P_1 \}$ &$(S_1P_1,S_1P_1)$&$(S_1P_1,S_1P_1)$ \\
				$\mathsf{add}\{S_1,S_2,P_1\}$ &$(S_2P_1,S_1P_1)$ & $(S_2P_1, S_1P_1)$ \\
				$\mathsf{add}\{P_1 \}$ &$(S_1P_1,S_2P_1)$ & $(P_1, P_1)$ \\
				$\mathsf{add}\{S_2,P_1\}$ &$(S_2P_1,S_2P_1)$ & $(S_2P_1, S_2P_1)$ \\
				$\mathsf{add}\{S_3,P_2\}$ & $(S_3P_2,S_3P_2)$  &  $(S_3P_2,S_3P_2)$ \\
				$\mathsf{add}\{S_2,S_3, P_2\}$ & $(S_3P_2,S_2P_2)$ &  $(S_3P_2,S_2P_2)$\\
				$\mathsf{add}\{P_2\}$ & $(S_2P_2,S_3P_2)$ &  $(P_2,P_2)$  \\
				$\mathsf{add}\{S_2,P_2\}$ &  $(S_2P_2,S_2P_2)$ &  $(S_2P_2,S_2P_2)$ \\
				
				$\mathsf{add}\{S_3,P_1\}$ & $(S_1S_3P_1,S_2P_1P_2)$ & $(S_3P_1, S_3P_1)$ \\
				$\mathsf{add}\{S_1,P_2\}$ & $(S_2P_1P_2,S_1S_3P_2)$& $(S_1P_2,S_1P_2)$ \\
				$\mathsf{add}\{S_2,P_1,P_2\}$ & $(S_2P_1P_2,S_2P_1P_2)$& $(S_2P_1P_2,S_2P_1P_2)$ \\
				$\mathsf{add}\{S_1,S_3,P_2\}$ & $(\Lambda,S_1S_3P_2)$ & $(S_1S_3P_2, S_1S_3P_2)$ \\
				$\mathsf{add}\{S_2,S_3,P_1,P_2\}$ &$(\Lambda,S_2P_1P_2)$ & $(S_3P_1P_2, S_2P_1P_2)$ \\
				$\mathsf{add}\{S_1,S_3,P_1\}$ &$(S_1S_3P_1,D\Lambda)$ & $(S_1S_3P_1, S_1S_3P_1)$ \\
				$\mathsf{add}\{S_1,S_2,P_1,P_2\}$ & $(S_2P_1P_2,D\Lambda)$ & $(S_2P_1P_2, S_1P_1P_2)$ \\
				$\operatorname{mod}\Lambda$ & $(\Lambda,D\Lambda)$ & $(\Lambda,D\Lambda)$ \\
				\hline
			\end{tabular}
		\end{table}
	\end{example}

	\section*{Acknowledgments}
	
	Hanpeng Gao is supported by  National Natural Science Foundation of China (No.12301041),
	Dajun  Liu  is supported by  National Natural Science Foundation of China (No.12101003) and
	the Natural Science Foundation of Anhui province (No.2108085QA07),
	Yu-Zhe Liu is supported by National Natural Science Foundation of China (Nos. 12561008, 12401042),
	Guizhou Provincial Major Project of Basic Research Program (Grant No. VZD[2026]001),
	Guizhou Provincial Basic Research Program (Natural Science) (Grant Nos. ZD[2025]085 and ZK[2024]YiBan066), 
	and Scientific Research Foundation of Guizhou University (Grant Nos. [2023]16).
	The authors would like to thank Dhunya Saito and Francesco Sentieri for helpful discussions.


\end{document}